\documentclass{amsart}

\usepackage{amssymb, amsmath, amsthm,stmaryrd,comment}
\usepackage[backref]{hyperref}
\usepackage[alphabetic,backrefs,lite]{amsrefs}
\newtheorem{lemma}{Lemma}[section]
\newtheorem{theorem}[lemma]{Theorem}

\newtheorem{cor}[lemma]{Corollary}
\newtheorem{conj}[lemma]{Conjecture}

\newtheorem{claim*}{Claim}
\newtheorem{thm}[lemma]{Theorem}
\newtheorem{defn}[lemma]{Definition}

\theoremstyle{definition}

\newtheorem{example}[lemma]{Example}

\newcommand{\cN}{\mathcal{N}}

\newcommand{\cB}{\mathcal{B}}

\newcommand{\PP}{{\mathbb P}}

\newcommand{\F}{{\mathbb F}}

\newcommand{\Q}{{\mathbb Q}}

\newcommand{\Z}{{\mathbb Z}}

\newcommand{\calB}{{\mathcal B}}

\newcommand{\calO}{{\mathcal O}}

\newcommand{\OO}{{\mathcal O}}

\newcommand{\fp}{\mathfrak{p}}
\newcommand{\mP}{\mathfrak{P}}
\newcommand{\ff}{\mathfrak{f}}
\DeclareMathOperator{\Tr}{Tr}

\DeclareMathOperator{\Gal}{Gal}

\DeclareMathOperator{\Norm}{ Norm}

\DeclareMathOperator{\ord}{ord}

\numberwithin{equation}{section}
\numberwithin{table}{section}

\title[$S$-Unit Equations
and the Asymptotic Fermat Conjecture]{$S$-Unit Equations
and the Asymptotic Fermat Conjecture over Number Fields}

\author{Ekin Ozman}
\author{Samir Siksek}

\address{Bogazici University\\
Department of Mathematics\\
Bebek, Istanbul, 34342 \\
Turkey}

\email{ekin.ozman@boun.edu.tr}

\address{Mathematics Institute\\
    University of Warwick\\
    CV4 7AL \\
    United Kingdom}

\email{s.siksek@warwick.ac.uk}

\thanks{Ozman is partially supported by TUBITAK Research Grant
117F045 and Max Planck Institute for Mathematics (MPIM) and would like to thank
MPIM for providing excellent research facilities.
Siksek is supported by
EPSRC grant \emph{Moduli of Elliptic curves and Classical Diophantine Problems}
(EP/S031537/1).}
\date{\today}

\begin{document}

	\maketitle

\begin{abstract}
Recent attempts at studying the Fermat equation
over number fields have uncovered an unexpected
and powerful connection with $S$-unit equations.
In this expository paper we explain this connection
and its implications for the asymptotic Fermat conjecture.
\end{abstract}

\begin{comment}
\ekin{	Tentative OUTLINE

	\begin{itemize}
	\item Intro (write at the end)- Section 1

	\item What is an S-unit, S-unit equation, why solutions are finite? Examples- Section 2

	\item Applications of solving S-unit equations, just name a few -Section 3
	\item Explain Fermat type in great detail, mention the papers used this technique, explain the technique, compare the technique with wiles method-  Section 3
	\item Examples of cases that FLT holds via S-units- Section 4
	\end{itemize}
	}
\end{comment}

\section{Introduction}
Every mathematician is familiar with
the statement of Fermat's Last Theorem,
proved by Wiles and Taylor \cite{Wiles}, \cite{TW} in 1994.
\begin{thm}[Wiles]
Let $p \ge 3$ be a prime. Then the only solutions to the
equation
\begin{equation}\label{eqn:Fermat}
x^p+y^p+z^p=0
\end{equation}
with $x$, $y$, $z \in \Q$ satisfy $xyz=0$.
\end{thm}
This survey is concerned with generalizations of Fermat's Last Theorem
where $\Q$ is replaced by a number field $K$, and also
with similar Fermat-type equations where $A, B, C$ are in the ring of
integers $\OO_K$ of $K$:
\begin{equation}\label{eqn:genFermat}
A x^p+B y^p+Cz^p=0, \qquad A x^p+B y^p=C z^2, \qquad
A x^p+B y^p=C z^3,
\end{equation}
again over number fields.
Interest in the Fermat equation \eqref{eqn:Fermat} over number fields
goes back to the 19th century and early 20th century.
Dickson \cite[pages 758 and 768]{Dickson} in his monumental
\emph{History of the Theory of Numbers}, surveys the early
history and mentions the efforts of
%Before Wiles,
%most attacks on the Fermat equation over number fields
%were in fact generalizations of Kummer's cyclotomic approach to the
%Fermat equation over $\Q$. It appears from Dickson's survey
Maillet (1897) and Furtw\"{a}ngler (1910)
who extended Kummer's cyclotomic approach to the Fermat
equation over $\Q(\zeta_p)$.
Later, Hao and Parry \cite{HP}
used the Kummer approach to prove several results concerning
the exponent $p$ Fermat equation \eqref{eqn:Fermat}
over a quadratic field $\Q(\sqrt{d})$
subject to the condition that the prime $p$ does not
divide the class number of $\Q(\sqrt{d}, \zeta_p)$.
The following theorem is due
to Kolyvagin \cite{Kolyvagin}, and is a beautiful example
of how far the cyclotomic approach can be pushed.
\begin{thm}[Kolyvagin]
Let $p \ge 5$ be a prime
and write $\zeta_p$ for a primitive $p$-th root of unity.
Let $x$, $y$, $z \in \Z[\zeta_p]$ satisfy \eqref{eqn:Fermat},
with $(1-\zeta_p) \nmid xyz$ (such a solution is
called a \lq first case solution\rq). Then $p^2 \mid (q^p-q)$
for all primes $q \le 89$ with $q \ne p$.
\end{thm}

Another historically popular approach is to fix a prime
exponent $p$ and consider points of low degree
(i.e.\ points defined over number fields of low degree)
 on the Fermat curve $x^p+y^p+z^p=0$.
For example, Gross and Rohrlich \cite{GR}
determine all points on
the Fermat curve $F_p \, : \,  x^p+y^p+z^p=0$ for
$p=3$, $5$, $7$, $11$ over all number fields $K$
of degree $\le (p-1)/2$ through studying the Mordell--Weil
group of the Jacobian of $F_p$.

After Wiles' proof of Fermat's Last Theorem using Galois representations
and modularity, others tried to extend this approach to various number
fields. The first result along these lines is due to Jarvis and Meekin
\cite{JM} stating that the only solutions to \eqref{eqn:Fermat}
with $x$, $y$, $z \in \Q(\sqrt{2})$ satisfy $xyz=0$.
This was later extended to some real quadratic fields
of small discriminant by Freitas and Siksek \cite{FSsmall},
and to some imaginary quadratic fields of small discriminant
by Tur\c{c}a\c{s} \cite{Turcas} (the later being conditional
on some standard conjectures in the Langlands programme).

\bigskip

Let $K$ be a number field. We say that a solution $(x,y,z) \in K^3$
to the Fermat equation \eqref{eqn:Fermat} is
\textbf{trivial} if $xyz=0$ and \textbf{non-trivial} if
$xyz \ne 0$. In this survey we are primarily concerned with the
following conjecture, which appears to have first
been formulated in \cite{FKS}.
\begin{conj}[The Asymptotic Fermat Conjecture]
Let $K$ be a number field, and suppose the primitive third root of unity, $\zeta_3 \notin K$.
There exists a constant $\mathcal{B}_K$ depending only on $K$ such that
for all primes $p>\mathcal{B}_K$ the only solutions to the Fermat equation
\eqref{eqn:Fermat} with $(x,y,z) \in K^3$ are the trivial solutions.
\end{conj}
We remark that the asymptotic Fermat conjecture follows
from a suitable version of the $ABC$-conjecture
over number fields \cite{Browkin}.

\noindent \textbf{Remarks.}
\begin{itemize}
\item Observe that for $p \ne 3$ we have $1^p+\zeta_3^p+\zeta_3^{2p}=0$.
For this reason it is necessary to exclude number fields
containing $\zeta_3$ in the statement of the conjecture.
\item We cannot expect the statement of Fermat's Last Theorem
to be true over every number field without modification. Indeed,
fix the exponent $p$ for now. The Fermat curve $F_p : x^p+y^p+z^p=0$
contains the rational point $(1:-1:0)$. Now take a line defined over $\Q$
through this point. This must intersect $F_p$ in a further $p-1$ points.
We see that $F_p$ has an infinite family of points defined
over number fields of degree $\le p-1$.
It therefore makes
sense to consider the Fermat equation over a given number field
asymptotically, i.e. for large exponents $p$.
\item Debarre and Klassen \cite{DebarreKlassen} suggest
that the only points on the degree $p$ Fermat curve \eqref{eqn:Fermat}
over number fields of degree $d \le p-2$
%are the trivial ones
%and permutations of $(1:\zeta_3:\zeta_3^2)$.
lie on the line $x+y+z=0$. Observe that  the six obvious points
$(1:-1:0)$, $(1:\zeta_3:\zeta_3^2)$, and their permutations,
do lie on this line.
\end{itemize}

\bigskip

We are grateful to the referees for many useful comments.

\section{The Modular Approach---An Example of Serre and Mazur}
\label{sec:SM}
As we shall see later, it is often possible to relate non-trivial
solutions to Fermat-type equations to solutions to certain
$S$-unit equations. In this section we sketch the
earliest instance of this phenomena, which is an example
due to Serre and Mazur, given in Serre's 1987 Duke
article where he formulated his famous modularity conjecture
\cite{SerreDuke}.
The sketch will be slightly technical, and the
reader unfamiliar with Galois representations and modularity
should feel free to skim through it. Good introductions
to the subject include \cite{bennett2016generalized}
and \cite{SiksekModular}.

\bigskip

Let $L$ be either $1$ or an odd prime. Let $(x,y,z) \in \Z^3$ be a
non-trivial solution
to the equation
\begin{equation}\label{eqn:SM}
x^p+y^p+L^r z^p=0
\end{equation}
where the exponent $p \neq L$ is a prime $\ge 5$ and $r$ is a non-negative integer.
Moreover we may (after suitable scaling and possible
rearrangement of the variables) suppose that $\gcd(x,y,Lz)=1$.
We suppose $r<p$ since we can absorb $p$-th powers of $L$
into $z^p$.
Note that we allow $L=1$ as we would like to include Fermat's Last
Theorem in our sketch.
We let $A$, $B$, $C$ be the three terms $x^p$, $y^p$, $L^r z^p$
arranged so that $A \equiv -1 \pmod{4}$ and $2 \mid B$. Let
$E^\prime$ be the Frey elliptic curve
\[
E^\prime \; : \; Y^2=X(X-A)(X+B).
\]
Serre studies the mod $p$ representation $\overline{\rho}_{E^\prime,p}$
of $E^\prime$,
which is irreducible by Mazur's isogeny theorem.
It follows from theorems of Ribet and Wiles that
the representation $\overline{\rho}_{E^\prime,p}$ arises
from a cuspidal newform $f$ with trivial character of weight $2$ and level
$N=2L$. If $L=1$ (the FLT case) then $N=2$. However, there are no newforms of
weight $2$ and level $2$, which gives
a contradiction, and so there are no non-trivial solutions
for $L=1$. The proof of Fermat's Last Theorem is complete
at this point. In fact there are no newforms of weight $2$
and levels $6$, $10$, $22$. Thus for $L=3$, $5$, $11$ we can also
conclude that there are no non-trivial solutions to
\eqref{eqn:SM}. However, it is easy to deduce
from the dimension formula for newform spaces
\cite[Proposition 15.1.1]{CohenII} that
there are newforms of weight $2$
and level $2L$
for all other odd prime values
$L=7,13,17,19,\dotsc$.
To progress we need to know
a little about the relationship between $E^\prime$
and the newform $f$. The newform $f$ has a $q$-expansion
\[
f=q+\sum_{n=1}^\infty c_n q^n.
\]
The coefficients $c_n$ generate a totally real field $K_f$
and in fact belong to the ring of integers $\OO$ of $K_f$.
There is some prime ideal $\varpi$ of $\OO$ dividing $p$
so that for any prime $\ell \nmid 2Lp$ the following relations hold
\[
\begin{cases}
a_\ell(E^\prime) \equiv c_\ell \pmod{\varpi} & \text{if $\ell \nmid xyz$}\\
\pm (\ell+1) \equiv c_\ell \pmod{\varpi} & \text{if
$\ell \mid xyz$}.
\end{cases}
\]
We do not know the elliptic curve $E^\prime$ as this depends
on a hypothetical solution $(x,y,z)$ to \eqref{eqn:SM}. However, given $\ell \nmid 2L$,
the trace $a_\ell(E^\prime)$ is an integer belonging
to the Hasse interval $[-2\sqrt{\ell},2\sqrt{\ell}]$.
It follows from the above congruences that $\varpi$ divides
\[
\beta_\ell:=\ell \cdot (\ell+1-c_\ell) \cdot (\ell+1+c_\ell)  \, \cdot
\prod_{-2\sqrt{\ell} \le a \le 2\sqrt{\ell}} (a-c_\ell).
\]
As $\varpi$ is a prime ideal dividing $p$ it follows that
$p \mid B_\ell$ where $B_\ell=\Norm_{K_f/\Q}(\beta_\ell)$.
This gives a bound for the exponent $p$ provided $B_\ell \ne 0$
or equivalently $\beta_\ell \ne 0$. Note that
if $c_\ell \notin \Q$ then $\beta_\ell \ne 0$.
If $K_f \ne \Q$ then there is a positive density of primes
$\ell$ such that $c_\ell \notin \Q$ and choosing any of these
with $\ell \nmid 2L$
gives a bound for $p$, and we will be content with that.
From now on our aim is to show that $p$ is bounded.
For example, there are newforms of weight $2$ and level $2 \times 37$
but these are irrational (a newform $f$ is
\textbf{irrational} if $K_f \ne \Q$ and
\textbf{rational} if $K_f=\Q$).
Thus the exponent $p$ is bounded for non-trivial solutions
to \eqref{eqn:SM} when $L=37$.
However, this is not the case for $L=7, 13, 17, 19, 23, 29,
31, 41, \dotsc$ where we do find rational newforms at levels $2L$.
We now ignore the irrational newforms (as they give a bound for
$p$) and focus on the rational ones.

\medskip

A theorem of Eichler and Shimura asserts that a rational weight
$2$ newform $f$ corresponds to an isogeny class
of elliptic curves $E$ defined over $\Q$.
This correspondence was made more precise by Carayol \cite{Carayol}
 who showed that the
level $N$ of $f$ is equal to the conductor of each $E$ in the
isogeny class. The correspondence
asserts that $a_\ell(E)=c_\ell$ for all primes $\ell \nmid N$.
We apply this to our rational newform $f$ of weight $2$ and level $N=2L$.
The earlier congruences become
\[
\begin{cases}
a_\ell(E^\prime) \equiv a_\ell(E) \pmod{p} & \text{if $\ell \nmid xyz$}\\
\pm (\ell+1) \equiv a_\ell(E) \pmod{p} & \text{if
$\ell \mid xyz$}.
\end{cases}
\]
Note that $E^\prime$ has full $2$-torsion and
thus $4 \mid \#E^\prime(\F_\ell)$ for all primes $\ell$
of good reduction.
However, $\# E^\prime(\F_\ell)=\ell+1-a_\ell(E^\prime)$.
Thus $a_\ell(E^\prime)$ belongs to the set
\[
T_\ell=\{a \in \Z \; : \; -2\sqrt{\ell} \le a \le 2\sqrt{\ell}, \qquad \ell+1 \equiv a \pmod{4}\}.
\]
This leads us to conclude that $p$ divides
\[
\gamma_\ell:=\ell \cdot (\ell+1-a_\ell(E)) \cdot (\ell+1+a_\ell(E))  \, \cdot
\prod_{ a \in T_\ell} (a-a_\ell(E))
\]
for any prime $\ell \nmid 2L$. If $\gamma_\ell$ is non-zero
for some $\ell \nmid 2L$ then
we have a bound for the exponent $p$ for non-trivial solutions to
\eqref{eqn:SM}. If $a_\ell(E) \notin T_\ell$ for some prime $\ell \nmid 2L$
then $\gamma_\ell$ is non-zero and we have a bound for $p$.
Thus we are reduced to the case where $a_\ell(E) \in T_\ell$
or equivalently $4 \mid \#E(\F_\ell)$, for all primes $\ell \nmid 2L$.
In that case, it follows from \cite[Lemma 7.5]{Haluk}
that $E$ is isogenous
to an elliptic curve with full $2$-torsion, and since $E$ is
really determined only up to isogeny we now suppose that $E$
has full $2$-torsion. It remains to determine, for which
odd primes $L$,
there is an elliptic curve $E/\Q$ with full $2$-torsion
and conductor $2L$. The answer is given by the following lemma.
\begin{lemma}\label{lem:FM}
Let $L$ be an odd prime. Then there is an elliptic curve
$E/\Q$ with full $2$-torsion and conductor $2L$
if and only if $L$ is a Mersenne or a Fermat prime and $L\ge 31$.
\end{lemma}
\begin{proof}
Such an $E$ necessarily has model
\[
E \; : \; Y^2=X(X-a)(X+b)
\]
with $a$, $b \in \Z$ and $ab(a+b) \ne 0$; indeed the discriminant is
$16 a^2 b^2 (a+b)^2$. Moreover we can choose $a$, $b$
so that this model is minimal away from $2$. Thus
\[
a^2 b^2(a+b)^2= 2^u L^v
\]
for some non-negative integers $u$, $v$. It follows that
\[
a=\pm 2^{u_1} L^{v_1}, \qquad b=\pm 2^{u_2} L^{v_2}, \qquad a+b=\pm2^{u_3} L^{v_3}.
\]
Thus
\begin{equation}\label{eqn:sunit2L}
\pm 2^{u_1}L^{v_1} \pm 2^{u_2} L^{v_2}= \pm 2^{u_3} L^{v_3}.
\end{equation}
This is an $S$-unit equation with $S=\{2,L\}$ ($S$-unit equations
are defined in Section~\ref{sec:Sunit}). It is an easy exercise
to conclude from this equation that $L$ is a Fermat or a Mersenne prime,
or that $v_1=v_2=v_3$. However if $v_1=v_2=v_3$ then the exponent
of $L$ in the conductor of $E$ is not $1$. Also for the Mersenne and
Fermat primes $L=3$, $5$, $7$ and $17$, the exponent of $2$ in the conductor of $E$
is not $1$. So we conclude that $L$ is a Mersenne or a Fermat prime
and $L \ge 31$.
\end{proof}

We have the following theorem.
\begin{thm}[Serre and Mazur]
Let $L$ be an odd prime. Suppose $L<31$, or $L$ is neither a
Mersenne
nor a Fermat prime. Then there is a constant $C_L$ such that
for all primes $p>C_L$ the only solutions $(x,y,z) \in \Z^3$
to the
equation \eqref{eqn:SM} are the trivial ones satisfying $xyz=0$.
\end{thm}
We note in passing that the equation $x^p+y^p+2^r z^p=0$ is much more difficult
due to the presence of the non-trivial solution $(x,y,z,r)=(1,1,-1,1)$
for all exponents $p$.
Thus no bound for the exponents $p$ of non-trivial solutions is possible.
Ribet \cite{Ribet2} and Darmon and Merel \cite{DM} showed that there are no
solutions apart from the trivial ones and $(x,y,z,r)=(1,1,-1,1)$ and $(-1,-1,1,1)$.

\section{Modular Approach---A General Sketch}\label{sec:gen}
Most modern attacks on Fermat-type equations \eqref{eqn:genFermat}
over a number field $K$ follow
the strategy of Serre and Mazur outlined in the previous section, which
we now briefly describe in more generality.
Again the reader should feel free to skim this section.
The steps are roughly as follows:
\begin{enumerate}
\item[(I)] Associate a Frey elliptic curve $E^\prime$ to a non-trivial solution
$(x,y,z)$.
\item[(II)] Show that the mod $p$ representation $\overline{\rho}_{E^\prime,p}$
 is irreducible.
No generalization of  Mazur's isogeny theorem is available over
number fields. However the desired irreducibility often follows for suitably
large $p$
from Merel's uniform boundedness theorem using the fact that
the Frey curve is close to being semistable.  This approach is
explained in \cite{irred}.
\item[(III)] Show that the $\overline{\rho}_{E^\prime,p}$ is modular
of parallel weight $2$ and level $\cN$
which is independent of the solution $(x,y,z)$ (the level
$\cN$ is an ideal of the ring of integers $\OO_K$).
Over totally real fields it is often possible to use the work of
Kisin, Gee, and others to achieve this. For example, in \cite{FLHS}
it is shown that for a given totally real field $K$ all but finitely
many $j$-invariants are modular. This is usually enough to
show that $\overline{\rho}_{E^\prime,p}$ is modular for $p$ sufficiently
large. Over general number fields we know much less about modularity
of elliptic curves and it is often necessary to assume
a version of Serre's modularity conjecture, as for example
in \cite{Haluk}, \cite{Turcas}.
\item[(IV)] Determine newforms of parallel weight $2$ and level $\cN$.
This is often a difficult step over number fields.
If there are none then one can conclude that there are no non-trivial
solutions. If they are all irrational then one should be able to
at least bound the exponent $p$.
\item[(V)] Instead of determining all newforms of parallel weight $2$ and level
$\cN$ one can focus on the rational newforms. Here there
is a conjectural generalization of the Eichler--Shimura theorem
which is often called the Eichler--Shimura conjecture.
If $K$ is totally real this simply says that a newform of
parallel weight $2$ and level $\cN$ corresponds to an
isogeny class of elliptic curves $E$ of conductor $\cN$,
and this conjecture is in fact known to be true (e.g. \cite{Hida})
if $\cN$ is not squarefull (i.e.\ there is a prime ideal $\mathfrak{q}$ with
$\ord_{\mathfrak{q}}(\cN)=1$). For a version of the Eichler--Shimura
conjecture over general number fields $K$ see \cite{Haluk}.
At any rate, assuming this conjecture, or relying on special cases
of the conjecture that are theorems, we know the existence of
an elliptic curve $E/K$ of conductor $\cN$ with
$\overline{\rho}_{E^\prime,p} \sim \overline{\rho}_{E,p}$.
It is usually possible to show that $E$ has the same torsion
structure as $E^\prime$.
\item[(VI)] So we would like to determine
all elliptic curves $E/K$ of condutor $\cN$ and
having a certain torsion structure.
This can be treated as a Diophantine problem.
For example, to determine all elliptic
curves $E/K$ of conductor $\cN$ with full $2$-torsion
it is enough to solve a certain $S$-unit equation
where $S$ is the set of prime ideals dividing $2 \cN$
(we will say more on that in Section~\ref{EC}).
Not every solution to the $S$-unit equation will
lead back to an elliptic curve of the right conductor.
For example, in the proof of Lemma~\ref{lem:FM}
we excluded solutions to the $S$-unit equation
\eqref{eqn:sunit2L} with $v_1=v_2=v_3$ as these
do not lead back to an elliptic curve of conductor $2L$.
Thus we are probably interested in all solutions to the $S$-unit
equation that satisfy further restrictive conditions.
\end{enumerate}

\section{$S$-Unit Equations}\label{sec:Sunit}
Let $K$ be a number field,  $\OO_K$ be its ring of integers and
$S$ be a finite set of primes ideals of $\OO_K$. In simplest terms, the notion
of $S$-unit generalizes the idea of a unit in $\OO_K$.
\begin{defn}
An \textbf{$S$-unit} is an element $\alpha$ in $K$ such
that the principal fractional ideal generated by $\alpha$
can be written as a product of the prime ideals in $S$.
In other words, the set of $S$-units $ \OO_S^*$ can be defined as:
\[
\OO_S^*=\{ \alpha \in K^* \; : \; \ord_\fp(\alpha)=0 \quad \text{for all
$\fp \notin S$}\}.
\]
Similarly  the set of $S$-integers in $K$ is \[
	\OO_S=\{ \alpha \in K^* \; : \; \ord_\fp(\alpha) \geq 0 \quad \text{for all
	$\fp \notin S$}\}.
		\]
\end{defn}

Note that $S$-units $\OO_S^*$ are units of the ring of $S$-integers $\OO_S$.

\begin{example}
Let $K=\Q$.
Every ideal of $\OO_K=\Z$ is principal, and prime ideals
are generated by primes. Thus we may think of $S$
as a finite set of primes $S=\{p_1,p_2,\dotsc,p_r\}$.
Then an $S$-unit of $K$ is a rational number $\frac{a}{b}$ such
that $a$ and $b$ are only divisible by the primes in $S$;
i.e.
\[
\OO_S^*=\{\pm p_1^{a_1} \cdots p_r^{a_r} \; : \; a_1,\dotsc,a_r \in \Z\}.
\]
%\begin{itemize}
%\item If
%$S=\{2\}$ then $\OO_S^*=\{\pm 2^r : r \in \Z\}$.
%\item If $S=\{2,L\}$ then $\OO_S^*=\{\pm 2^r L^s : r,~s \in \Z\}$.
%\end{itemize}
\end{example}

\begin{example}
Let $K=\Q(\sqrt{5})$, whence $\calO_K =\Z[\frac{1+\sqrt{5}}{2}]$.
\begin{itemize}
\item
If $S=\emptyset$ then
$ \OO_S^*=\left\{ \pm \left(\frac{1+\sqrt{5}}{2}\right)^r : r \in \Z\right\}$.
\item	If $S=\{2 \OO_K\}$ then $ \OO_S^*=\left\{ \pm 2^r \left(\frac{1+\sqrt{5}}{2}\right)^s : r,~s \in \Z\right\}.$
\end{itemize}
	\end{example}

If $S$ and $T$ are sets of prime ideals and $T \subseteq S$
then $\OO_T^*$ is a subgroup of $\OO_S^*$.
Observe that the unit group $\OO_K^*$ is precisely
$\OO_\emptyset^*$. Thus $\OO_K^*$ is a subgroup of
$\OO_S^*$, and every unit is indeed an $S$-unit.
Many facts concerning units have generalizations
to $S$-units.
\begin{theorem}[Dirichet's $S$-Unit Theorem]
%The set of $S$-units, $\OO_S$,  is a multiplicative group containing
%$\calO_K^*$. Morever
The $S$-unit group
$\OO_S^*$ is finitely generated  with rank equal to $r_1+r_2 + \#S-1$,
where $(r_1,r_2)$ is the signature of $K$.
\end{theorem}
Observe that letting $S=\emptyset$ allows us to see the
Dirichlet's unit theorem as a special case of Dirichlet's
$S$-unit theorem.

\begin{defn}
Let $K$ be a number field and $S$ a finite set of prime
ideals of $\OO_K$.
The \textbf{$S$-unit equation} is the equation
\begin{equation}\label{eqn:sunit}
\lambda+\mu=1, \qquad \lambda,~\mu \in \OO_S^*.
\end{equation}
If $S=\emptyset$ so $\OO_S^*=\OO_K^*$ then this is called
the \textbf{unit equation}.
\end{defn}

\begin{theorem}[Siegel 1921, Parry 1950]\label{thm:Siegel}
Let $K$ be a number field and $S$ a finite set of prime
ideals of $\OO_K$.
The $S$-unit equation \eqref{eqn:sunit} has only finitely many solutions.
\end{theorem}
The original proofs due to Siegel and Parry are non-effective.
Later on, Baker's
theory of linear forms in logarithms gave effective
though very large bounds for the solutions. In his 1989 PhD
thesis Benne de Weger \cite{deWeger} showed how these bounds
can be combined with the LLL algorithm to give a practical
method for solving such equations. Variants of de Weger's
algorithm can be found in Smart's book \cite{Smart}
and also in \cite{sunit} and \cite{Matschke}.

\begin{example}\label{ex:Smart}
We illustrate the practicality of the algorithm
of de Weger and its variants through the following
example.
Let $F=\Q(\zeta_{16})$ where $\zeta_{16}$ is a primitive
$16$-th root of unity. Then $F$ is a totally complex
number field of degree $8$.
Let $\mathfrak{p}=(1-\zeta_{16}) \cdot \OO_F$;
this is the unique prime above $2$.
Let $S^\prime=\{\mathfrak{p}\}$.
%Then $2 \OO_K= \PP^8$. It follows that $S=T=\{\mP\}$.
Smart
\cite{Smart2}
determines the solutions to the
equation $\lambda+\mu=1$ with $\lambda$, $\mu \in \OO_{S^\prime}^*$
and finds that there are precisely
$795$ solutions $(\lambda,\mu)$---too many to enumerate here!

Let $K=F^+=\Q(\zeta_{16}+\zeta_{16}^{-1})=\Q\left(\sqrt{2+\sqrt{2}}\right)$
be the maximal totally real subfield of $F$, which has degree $4$.
Let $\mathfrak{P}=\sqrt{2+\sqrt{2}} \cdot \OO_{K}$
be the unique prime above $2$ in $\OO_{K}$,
and let $S=\{\mathfrak{P}\}$. The
$S$-unit equation $\lambda+\mu=1$
with $\lambda$, $\mu \in \OO_{S}^*$ has $585$
solutions. Of course this is a subset
of the $795$ solutions to
%\eqref{eqn:sunit}.
the $S^\prime$-unit equation in $F$.
%It turns out that the largest possible
%value of
%$\max\{\lvert \ord_\mP(\lambda)\rvert,\lvert\ord_\mP(\mu)\rvert\}$ is $22$, which is smaller
%than $4 \ord_\mP(2)=32$. By Theorem~\ref{thm:FermatGen},
%assuming Conjectures~\ref{conj:Serre} and~\ref{conj:ES},
%the asymptotic Fermat's Last Theorem holds for $K$.
\end{example}

\begin{example}
Let $K$ be a number field in which there is a degree $1$ prime
$\mathfrak{P}$ above $2$ (i.e. the residue field $\OO_K/\mathfrak{P}=\F_2$).
Let $S$ be a finite set of prime ideals of odd norm.
If $\lambda$, $\mu \in \OO_S^*$ then
$\lambda$, $\mu \equiv 1 \pmod{\mathfrak{P}}$ and so
$\lambda+\mu \equiv 0 \pmod{\mathfrak{P}}$.
Thus the $S$-unit equation \eqref{eqn:sunit} has no solutions.
\end{example}

Before de Weger the most promising method for solving $S$-unit
equations was Skolem's $p$-adic method (now often called
Chabauty--Coleman--Skolem). This method still has a lot of promise,
as the following recent and
beautiful theorem of Nicholas Triantafillou \cite{Triantafillou} shows.
\begin{thm}[Triantafillou]
Let $K$ be a number field. Suppose that $3 \nmid [K : \Q]$
and $3$ splits completely in K. Then there is no
solution to the unit equation in $K$. In other words,
there is no pair $\lambda$, $\mu \in \OO_K^*$
such that $\lambda+\mu=1$.
\end{thm}
\begin{proof}
We can't resist giving an exposition of Triantafillou's elegant
argument. Let $K$ be a number field in which $3$ splits completely,
and write $3 \OO_K = \fp_1 \cdots \fp_n$ where $n=[K:\Q]$
and the $\fp_j$ are distinct prime ideals with residue field $\F_3$.
Let $\theta \in \OO_K^*$. Then $\theta \equiv \pm 1 \pmod {\fp_j}$
and hence $\theta^2 \equiv 1 \pmod{\fp_j}$ for all $j$.
Thus $\theta^2 \equiv 1 \pmod{3 \OO_K}$.

Now let $\lambda$, $\mu \in \OO_K^*$ satisfy $\lambda+\mu=1$.
By the above
\[
\lambda^2 \equiv 1 \pmod{3\OO_K}, \qquad
(\lambda-1)^2=(-\mu)^2 \equiv 1 \pmod{3 \OO_K}.
\]
Hence $2\lambda-1=\lambda^2-(\lambda-1)^2 \equiv 0 \pmod{3 \OO_K}$,
so $\lambda \equiv -1 \pmod{3\OO_K}$. We write $\lambda=-1+3\phi$
with $\phi \in \OO_K$. Let
$\phi_1,\dotsc,\phi_n$ be the images of $\phi$ under the
$n$ embeddings $K \hookrightarrow \overline{\Q}$. As $\lambda$
is a unit
\[
\pm 1=\Norm(\lambda)=(-1+3\phi_1) \cdots (-1+3\phi_n)
\equiv (-1)^n + (-1)^{n-1} \cdot 3 \Tr(\phi) \pmod{9}.
\]
By considering all the choices for $\pm 1$ and $(-1)^n$,
we obtain
$3 \Tr(\phi) \equiv -2$, $2$ or $0
\pmod{9}$. The first two are plainly  impossible and so
$\Tr(\phi) \equiv 0 \pmod{3}$.

However $\mu=1-\lambda=2-3\phi=-1+3(1-\phi)$ is also a unit.
Thus by the above, $\Tr(1-\phi) \equiv 0  \pmod{3}$. But
$\Tr(1-\phi)=n-\Tr(\phi)$. Therefore $n \equiv 0 \pmod{3}$
completing the proof.
\begin{comment}
\[
\pm 1=\Norm(\mu)=(2-3\phi_1) \cdots (2-3\phi_n)
\equiv 2^n-2^{n-1} \cdot 3 \cdot \Tr(\phi)
\equiv 2^n \pmod{9}.
\]
But one easily checks that $2^n \equiv \pm 1 \pmod{9}$ if and only if
$n \equiv 0 \pmod{3}$, giving the theorem.
\end{comment}
\end{proof}

Perhaps the most elegant theorem on $S$-unit equations is the following
result due to Evertse \cite{Evertse}.
\begin{thm}[Evertse]
Let $(r_1,r_2)$ be the signature of $K$ and let $S$
be a finite set of prime ideals of $\OO_K$. Then
the $S$-unit equation \eqref{eqn:sunit} has at most
$3 \times 7^{3 r_1+4r_2+2\#S}$ solutions.
\end{thm}
For extensive surveys of results on $S$-unit equations, see
\cite{EvertseGyory} and the introduction of \cite{BB}.
Nowadays the $S$-unit equation is often viewed
as $S$-integral points on $\PP^1 \setminus \{0,1,\infty\}$,
allowing for a variety of high-powered approaches
from arithmetic geometry to be applied, e.g.\
\cite{Kim},
\cite{LV},
\cite{Triantafillou2}.

\section{S-Unit Equations and Elliptic Curves} \label{EC}
In this section we explore more fully the relationship between solutions
to $S$-unit equations and certain families of elliptic curves.
A theorem of Shafarevich asserts that given a finite set of
prime ideals $S$ in the ring of integers $\OO_K$
of a number field $K$,
there are only finitely many elliptic curves $E/K$ with good reduction
outside $S$. For illustration we consider a special case of this
problem where $K=\Q$ and $E$ is assumed to have a point
of order $2$. There is no loss of
generality in supposing that $2 \in S$.
We write $S=\{2,p_1,p_2,\dotsc,p_k\}$ where $p_1,\dotsc,p_k$
are distinct odd primes. We may suppose that $E$ has a model
of the form
\[
E \; : \; Y^2=X(X^2+aX+b)
\]
where $a$, $b$ are rational integers, and the discriminant
$\Delta=16b^2(a^2-4b) \ne 0$. Moreover, we can choose $a$, $b$
so that this model is minimal away from $2$. As $E$
has good reduction away from $S$ we see that
\[
b^2(a^2-4b)
=\pm 2^{\alpha_0}p_1^{\alpha_1} \ldots
p_k^{\alpha_k}
\]
 where $\alpha_i$ are nonnegative integers. Then
\[
b=\pm
2^{\beta_0} p_1^{\beta_1}\ldots p_k^{\beta_k}, \qquad
a^2-4b= \pm
2^{\alpha_0-2\beta_0}p_1^{\alpha_1-2\beta_1}p_2^{\alpha_2-2\beta_2}\ldots
p_k^{\alpha_k-2\beta_k}
\]
for some integers $0 \le \beta_i \le \alpha_i$.
Note that this gives a solution to the equation
$x+y=z^2$ with
\[
\begin{cases}
x=\pm 2^{\alpha_0-2\beta_0}p_1^{\alpha_1-2\beta_1}p_2^{\alpha_2-2\beta_2}\ldots
p_k^{\alpha_k-2\beta_k} \in \OO_S^*, \\
y= 4b=\pm 2^{\beta_0+2} p_1^{\beta_1}\ldots
p_k^{\beta_k} \in \OO_S^*,\\
z=a \in \Z.
\end{cases}
\]
More generally the task of determining elliptic curves
with a point of order $2$ over a number field $K$
and a good reduction outside a finite set of prime ideals $T$
reduces to solving an equation of the form
\begin{equation}\label{eqn:square}
 x+y=z^2, \qquad x,~y \in \OO_S^*, \qquad z \in K
\end{equation}
where $S$ is a suitable enlargement of $T$ that
takes account of the class group of $K$.
An algorithm for solving equations of the form
\eqref{eqn:square} is given in de Weger's thesis \cite{deWeger}.
See also \cite{BGR}.

%These solutions arise naturally when an effective version of the  theorem of
%Shafarevich wanted to be made. This is the theorem of Shafereich on the
%finiteness of isomorphism classes of elliptic curves over a number field $K$
%with good reduction outside a given finite set of primes. Here are the details
%of this argument: Say we want to find all elliptic curves $E$ over the rational
%numbers $\Q$ with nontrivial rational $2$-torsion and good reduction outside
%the finite set of primes $S=\{p_1, p_2, \ldots, p_k\}.$ We study $E/\Q$ via the
%following Weierstrass Equation $$E: y^2=x^3+ax^2+bx$$ where $a,b \in \Q,$ by
%assumption on good reduction $b^2(a^2-4b)=\pm 2^{\alpha_0}p_1^{\alpha_1} \ldots
%p_k^{\alpha_k}$ where $\alpha_i$ are nonnegative integers. Let $b=\pm
%2^{\beta_0} p_1^{\alpha_1}\ldots p_k^{\alpha_k}$. Then $a^2-4b= \pm
%2^{\alpha_0-2\beta_0}p_1^{\alpha_1-beta_1}p_2^{\alpha_2-\beta_2}\ldots
%p_k^{\alpha_k-\beta_k}$ and this gives an equation of the form $x+y=z^2$ with
%$x=\pm 2^{\alpha_0-2\beta_0}p_1^{\alpha_1-beta_1}p_2^{\alpha_2-\beta_2}\ldots
%p_k^{\alpha_k-\beta_k}, y= 4b=\pm 2^{\beta_0+2} p_1^{\alpha_1}\ldots
%p_k^{\alpha_k}, z=a$ and $S=\{2,p_1, \ldots, p_k\}. $

\bigskip

We now look at a similar problem that arises in the context
of understanding Fermat-type equations over number fields.
Let $K$ be a number field. An elliptic curve $E/K$
is said to have \textbf{potentially good reduction}
at a prime ideal $\mathfrak{q}$ of $\OO_K$ if there is a finite
extension $L/K$ so that $E/L$ has good reduction
at every prime ideal $\mathfrak{q}^\prime$ of $\OO_L$ above $\mathfrak{q}$.
It is possible to show that $E/K$ has potentially good
reduction at $\mathfrak{q}$ if and only if $\ord_\mathfrak{q}(j(E)) \ge 0$
where $j(E)$ is the $j$-invariant of $E$.
Now let $S$ be a finite set of prime ideals of $\OO_K$.
We are interested in the set $\mathcal{E}_S$ of elliptic
curves $E/K$ with full $2$-torsion and potentially
good reduction outside $S$. Here we suppose that $S$
includes all the prime ideals of $\OO_K$ above $2$.
We follow the treatment in \cite{H}.
By assumption the elliptic curves we are
dealing with are of the form
\begin{equation}\label{eqn:twotorsion}
E \; :\;  Y^2=(X-a_1)(X-a_2)(X-a_3)
\end{equation}
 where the $a_i \in K$ are
distinct.
Let $\lambda=(a_3-a_1)/(a_2-a_1) \in \PP^1(K) - \{0,1,\infty\}$.
This is called the \textbf{$\lambda$-invariant} of $E$.

%Similarly any
%$\mathbb P^1(K)-\{0,1,\infty\}$ can be written as the ratio of three distinct
%elements $a_1, a_2, a_3 \in K$ hence can be associated with an elliptic curve
%over $K$ with full two torsion.

\begin{lemma}\label{lem51}
Let $\mathcal{S}_3$ be the symmetric group on three elements. The action of
$\mathcal{S}_3$ on $\{a_1,a_2,a_3\}$ can be extended to $\mathbb
P^1(K)-\{0,1,\infty\}$.  Under this action the orbit of   $\lambda=(a_3-a_1)/(a_2-a_1) \in \mathbb{P}^1(K)-\{0,1,\infty\}$ is
\begin{equation}\label{eqn:lambdas}
\left\{\lambda,\;
\frac{1}{\lambda},
\; 1-\lambda,\;
\frac{1}{(1-\lambda)},
\;
\frac{\lambda}{(\lambda-1)},
\;
\frac{(\lambda-1)}{\lambda} \right\}.
\end{equation}
\end{lemma}
\begin{proof}
This is a straightforward computation. For example if
$\sigma \in \mathcal{S}_3$ is the transposition $(1,2)$
then it swaps $a_1$, $a_2$ and keeps $a_3$ fixed. Hence
\[
\sigma(\lambda)=(a_3-a_2)/(a_1-a_2)=1-\lambda.
\]
%Since the $a_i$ are distinct $[a_3-a_1: a_2-a_1] \neq 0$, $1$, $\infty$.
%Let $\sigma \in \mathcal{S}_3$ then $[\sigma(a_3-a_1) :
%\sigma(a_2-a_1)] \neq 1, 0 \; \text{or} \; \infty$ either. Let $\sigma= (12)$,
%then $\sigma (\lambda)= \frac{a_3-a_2}{a_1-a_2}=1-\lambda$. Similarly
%$\sigma(\lambda)=\frac{1}{\lambda}$ if $\sigma = (23)$,
%$\sigma(\lambda)=\frac{\lambda}{\lambda-1}$ if $\sigma = (13)$ ,
%$\sigma(\lambda)=\frac{1}{1-\lambda}$ if $\sigma = (123)$ and
%		$\sigma(\lambda)=\frac{\lambda-1}{\lambda}$ if $\sigma = (132)$. Hence we get the claimed orbit for $\lambda$.
\end{proof}
From now on we think of $\mathcal{S}_3$	as acting on $\PP^1(K)-\{0,1,\infty\}$,
via the six transformations $\lambda \mapsto \lambda$,
$\lambda \mapsto 1/\lambda$,
$\lambda \mapsto 1-\lambda$, \dots
\begin{lemma}\label{lem:inv}
The set of $\lambda$-invariants $\mathbb{P}^1(K)-\{0,1,\infty\}$,
up to equivalence
 under the action of $\mathcal {S}_3$,
is in  one to one correspondence with the set of
elliptic curves over $K$ with full two torsion up to isomorphism
over $\overline{K}$.
\end{lemma}
\begin{proof}
This is essentially Proposition III.1.7 in Silverman's book \cite{Sil}.
The correspondence is induced
by the association $E \mapsto \lambda=(a_3-a_1)/(a_2-a_1)$
where $E$ has the form \eqref{eqn:twotorsion}.
The inverse is given by sending the class of $\lambda \in \PP^1(K) - \{0,1,
\infty\}$ to the $\overline{K}$-isomorphism class of the
Legendre elliptic curve
\[
E_\lambda \; : \; Y^2=X(X-1)(X-\lambda).
\]
\end{proof}
%\begin{proof} This follows from the previous lemma and the fact that any elliptic curve over $K$ with full two torsion is isomorphic to the elliptic curve given by the Legendre form $E_{\lambda}:y^2= x(x-1)(x-\lambda)$ \cite{Sil}.
%\end{proof}

Let
\[
W_S=\{(\lambda, \mu) : \lambda+\mu=1, \lambda, \mu \in \OO_S^*\}
\]
 be
the set of solutions of the $S$-unit equation \eqref{eqn:sunit}.
Recall that $\mathcal{E}_S$ is the set of
elliptic curves over $K$ with full $2$-torsion and having potentially
good reduction outside $S$. If $E_1, E_2$ are in $\mathcal{E}_S$ and
isomorphic over the algebraic closure of $K$ then we say that $E_1, E_2$ are
\textbf{equivalent}.
\begin{lemma}\label{lem:correspondence}
Suppose $S$ is a finite set of prime ideals of $\OO_K$
that includes all the primes above $2$.
Then
$\mathcal{S}_3$ acts on $W_S$ via $\pi(\lambda, \mu)=(\pi(\lambda),
1-\pi(\lambda))$; here $\pi(\lambda)$ denotes the image of $\lambda$
under $\pi$ as in Lemma~\ref{lem51}.
Moreover the $\mathcal{S}_3$-orbits in $W_S$ are in bijection with the
equivalence classes in $\mathcal{E}_S$.
\end{lemma}
\begin{proof}
This is essentially routine computation; for full details see
\cite[Section 5]{H}. The bijection is induced by
the maps in Lemma~\ref{lem:inv}.
\end{proof}

\section{$S$-Unit Equations and Fermat}
In this section we state a theorem that relates the Fermat equation
over totally real fields to $S$-unit equations, following \cite{FS}.
Generalizations to fields with complex embeddings are known
and we discuss them in later sections, but the statement is
easier in the totally real setting.
In some cases we will need the Eichler--Shimura conjecture which
we now state.
\begin{conj}[\lq\lq Eichler--Shimura\rq\rq]\label{conj:ES}
Let $K$ be a totally real field. Let $\ff$ be a Hilbert newform over $K$
of level $\cN$ and parallel weight $2$, and
rational Hecke eigenvalues.
Then there is an elliptic curve $E_\ff/K$ with conductor $\cN$
having the same $\mathrm{L}$-function as $\ff$.
\end{conj}

Let $K$
be a totally real field, and let
\begin{equation}\label{eqn:ST}
\begin{gathered}
S=\{ \mP \; :\; \text{$\mP$ is a prime ideal of $\OO_K$ above $2$}\}, \\
T=\{ \mP \in S \; : \;
f(\mP/2)=1\},
\qquad
U=\{ \mP \in S \; : \;
3 \nmid \ord_\mP(2) \}.
\end{gathered}
\end{equation}
Here $f(\mP/2)$ denotes the residual degree of $\mP$.
We need an assumption, which we refer to as (ES):
\[ \label{ES}
\text{\bf (ES)} \qquad
\left\{
\begin{array}{lll}
\text{either $[K:\Q]$ is odd;}\\
\text{or $T \ne \emptyset$;}\\
\text{or Conjecture~\ref{conj:ES} holds for $K$.}
\end{array}
\right.
\]
\begin{thm}[Freitas and Siksek]\label{thm:FermatGen}
Let $K$ be a totally real field satisfying
 (ES). Let $S$, $T$ and $U$ be as in \eqref{eqn:ST}.
Write $\OO_S^*$ for the group of $S$-units of $K$.
Suppose that for every solution $(\lambda,\mu)$ to the $S$-unit equation
\eqref{eqn:sunit}
%\begin{equation}\label{eqn:sunit}
%\lambda+\mu=1, \qquad \lambda,\, \mu \in \OO_S^* \, .
%\end{equation}
there is
\begin{enumerate}
\item[(A)] either some $\mP \in T$ that satisfies
$\max\{ \lvert \ord_{\mP} (\lambda) \rvert, \lvert \ord_{\mP}(\mu) \rvert \}
\le 4 \ord_{\mP}(2)$,
\item[(B)] or some $\mP \in U$ that satisfies both
$\max\{ \lvert \ord_{\mP} (\lambda) \rvert, \lvert \ord_{\mP}(\mu) \rvert \}
\le 4 \ord_{\mP}(2)$, and
$\ord_{\mP}(\lambda \mu) \equiv \ord_{\mP}(2) \pmod{3}$.
\end{enumerate}
Then the asymptotic Fermat conjecture holds over
$K$.
\end{thm}
\begin{proof}[Proof Sketch]
The proof largely follows the strategy sketched in Sections~\ref{sec:SM}
and~\ref{sec:gen}. Write $E$ for the Frey curve associated
to a non-trivial solution to the generalized Fermat equation \eqref{eqn:Fermat}.
The strategy relates $\overline{\rho}_{E,p}$ to $\overline{\rho}_{F,p}$
where $F$ is an elliptic curve defined over $K$ with full $2$-torsion
and conductor $\mathcal{N}$ which does not depend on the solution
to the Fermat equation but only on the field $K$.
Inspired by ideas of Kraus \cite{Kraus}, and of Bennett and Skinner
\cite{BennettSkinner},
Freitas and Siksek study the possibilities for the image
of inertia $\overline{\rho}_{E,p}(I_\mP)$.
Since the representations $\overline{\rho}_{E,p}$
and $\overline{\rho}_{F,p}$ are isomorphic this
yields information about the elliptic curve $F$.
In particular they  deduce that $F$
has potentially good reduction at all primes outside $S$.
Lemma~\ref{lem:correspondence} relates $F$ to a solution
$(\lambda,\mu)$ of the $S$-unit equation~\eqref{eqn:sunit}.
The theorem follows from examining the possibilities for
$\overline{\rho}_{E,p}(I_\mP)$ and $\overline{\rho}_{F,p}(I_\mP)$
at the primes $\mP \in T$, $U$ and relating these to the solution
$(\lambda,\mu)$ of the $S$-unit equation \eqref{eqn:sunit}
corresponding to $F$. If either of hypotheses (A), (B) of the theorem
is satisfied then there will exist a prime $\mP$
such that
$\overline{\rho}_{E,p}(I_\mP) \not \cong \overline{\rho}_{F,p}(I_\mP)$,
and therefore the representations
$\overline{\rho}_{E,p}$ and $\overline{\rho}_{F,p}$
are non-isomorphic, giving a contradiction.
\end{proof}
We point out that a generalization of Theorem~\ref{thm:FermatGen}
to  general number fields is given by \c{S}eng\"{u}n
and Siksek \cite{Haluk}, assuming standard conjectures
stated in the following section.

\begin{example}
Let $K=\Q(\zeta_{16})^+=\Q\left(\sqrt{2+\sqrt{2}}\right)$.
This is a degree $4$ totally real field in which  $2$ is totally ramified:
$2 \OO_K=\mP^4$ where $\mP=\sqrt{2+\sqrt{2}} \cdot \OO_K$.
In particular, $S=T=\{\mP\}$ in the above notation.
As stated in Example~\ref{ex:Smart} the $S$-unit equation
\eqref{eqn:sunit} has $585$ solutions. It turns out that they
all satisfy condition (A) of the theorem. Hence
the asymptotic Fermat conjecture holds for $K$.
\end{example}

Through a detailed study of solutions to $S$-unit equations
over real quadratic fields, Freitas and Siksek \cite{FS}
prove the following, which in essence says that
the asymptotic Fermat conjecture holds for almost
all real quadratic fields.
\begin{thm}[Freitas and Siksek]\label{thm:density}
%Assume the Eichler--Shimura conjecture.
Let $\mathbb{N}^{\mathrm{sf}}$ denote the set of squarefree
natural numbers $>1$. Let $\mathcal{F}$ be the subset of $d \in \mathbb{N}^{\mathrm{sf}}$
for which the asymptotic Fermat conjecture holds over $\Q(\sqrt{d})$.
Then
\[
\liminf_{X \rightarrow \infty}
\frac{\# \{ d \in \mathcal{F} \; :\;  d \le X \}}{
\# \{d \in
\mathbb{N}^{\mathrm{sf}} \; : \; d \le X\}} \ge 5/6.
\]
If we assume the Eichler--Shimura conjecture then
\[
\lim_{X \rightarrow \infty}
\frac{\# \{ d \in \mathcal{F} \; :\;  d \le X \}}{
\#\{d \in
\mathbb{N}^{\mathrm{sf}} \; : \; d \le X\}} =1.
\]
\end{thm}

\subsection{$S$-Unit Equations and $\Z_\ell$-Layers}
In two recent works \cite{FKS} and \cite{FKS2}, Freitas, Kraus
and Siksek prove the asymptotic Fermat conjecture for
the layers of various cyclotomic $\Z_\ell$-extensions of $\Q$.
We first introduce these extensions. Let $\ell$ be
a rational prime. For now let $\ell$ be odd and $n \ge 1$.
The cyclotomic field $\Q(\zeta_{\ell^{n+1}})$ has
a unique subfield of degree $\ell^n$ which we denote
by $\Q_{n,\ell}$. This is a cyclic, totally real
extension of $\Q$ with Galois group $\Z/\ell^n \Z$.
Clearly $\Q_{n,\ell}$ is a subfield of $\Q_{n+1,\ell}$.
The union of these fields is denoted
\[
\Q_{\infty,\ell}=\bigcup_{n=1}^\infty \Q_{n,\ell}
\]
and has Galois group isomorphic to $\Z_\ell$.
This is called the cyclotomic $\Z_\ell$-extension of $\Q$,
and the field $\Q_{n,\ell}$ is called the $n$-th layer
of $\Q_{\infty,\ell}$.

For $\ell=2$ all the above is true with a small adjustment:
we take $\Q_{n,2}=\Q(\zeta_{2^{n+2}})^+$.
In \cite{FKS} the following theorem is proven.
\begin{thm}
The asymptotic Fermat conjecture is true for $\Q_{n,2}$.
\end{thm}
\begin{proof}[Proof Sketch]
Write $K=\Q_{n,2}$. Then $2$ is totally ramified
in $\OO_K$ and we let $\mP$ be the unique prime
above $2$. In the notation of Theorem~\ref{thm:FermatGen},
$S=T=\{\mP\}$. The key to the proof is to show that
every solution $(\lambda,\mu)$ to the $S$-unit equation
\eqref{eqn:sunit} satisfies condition (A) of Theorem~\ref{thm:FermatGen}.
Let $(\lambda,\mu)$ be a solution \eqref{eqn:sunit}.
Write
\[
m_{\lambda,\mu}:=
\max\{\lvert \ord_{\mP}(\lambda) \rvert, \lvert \ord_{\mP}(\mu) \rvert \};
\]
this is the quantity appearing in criterion (A) of Theorem~\ref{thm:FermatGen}.
%If $n_{\lambda,\mu}=0$ then
%$\lambda$, $\mu$ are both units
%and $\lambda \equiv \mu \equiv 1 \pmod{\mP}$ which contradicts
%$\lambda+\mu=1$. Hence $n_{\lambda,\mu}>0$.
%at least one of $\lambda$, $\mu$ has a
%non-zero valuation at $\mP$.
Suppose
\begin{equation}\label{eqn:assumption}
m_{\lambda,\mu}>2 \ord_{\mP}(2).
\end{equation}
The $\mathcal{S}_3$-action does not affect the
value of $m_{\lambda,\mu}$, and
by considering this action on
$(\lambda,\mu)$ we may suppose that
$\ord_\mP(\mu)=0$ and $\ord_\mP(\lambda)=m_{\lambda,\mu}$.
Then $\mu \in \OO_K^*$ and
$\mu=1-\lambda \equiv 1 \pmod{4}$ by assumption~\eqref{eqn:assumption}.
It follows from this that the extension $K(\sqrt{\mu})/K$
is unramified at $\mP$. Since $\mu$ is a unit, this extension
is unramified at all odd primes. Thus $K(\sqrt{\mu})/K$
is unramified at all the finite places. We now shall need
a theorem due to Iwasawa which asserts that
$K=\Q_{n,2}$ has odd narrow class number. Thus
$K(\sqrt{\mu})=K$ and so $\mu$ is a square.
We write $\mu=\delta^2$ where $\delta \in \OO_K^*$. Thus
\[
(1+\delta)(1-\delta)=1-\mu=\lambda.
\]
Hence
\[
\lambda=\lambda_1 \lambda_2, \qquad \lambda_1=1+\delta, \qquad \lambda_2=1-\delta.
\]
Now
\begin{equation}\label{eqn:mult}
\lambda_1+\lambda_2=2, \qquad \lambda_1-\lambda_2=2\delta.
\end{equation}
It follows easily that one of the $\ord_\mP(\lambda_i)$ is $m-\ord_\mP(2)$
and the other is $\ord_{\mP}(2)$,
where $m=m_{\lambda,\mu}=\ord_\mP(\lambda)$.
By swapping $\delta$ and $-\delta$ if necessary,
we may suppose $\ord_\mP(\lambda_1)=m-\ord_\mP(2)$ and
$\ord_\mP(\lambda_2)=\ord_{\mP}(2)$.
Multiplying the two equations in \eqref{eqn:mult}, dividing
by $\lambda_2^2$ and rearranging we obtain
\[
\lambda^\prime+\mu^\prime=1,
\qquad \lambda^\prime=\frac{\lambda_1^2}{\lambda_2^2},
\qquad \mu^\prime=\frac{-4 \delta}{\lambda_2^2}.
\]
Observe that $\lambda^\prime$, $\mu^\prime \in \OO_S^*$
so we obtain another solution to \eqref{eqn:sunit}.
Moreover,
\[
m_{\lambda^\prime,\mu^\prime}=2m_{\lambda,\mu}-2\ord_\mP(2)> m_{\lambda,\mu},
\]
where the last inequality follows from \eqref{eqn:assumption}.
This shows that the solution $(\lambda^\prime,\mu^\prime)$
is different from $(\lambda,\mu)$ and also satisfies
\eqref{eqn:assumption}. Repeating the argument allows us to construct
infinitely many solutions to the $S$-unit equation
contradicting Siegel's theorem (Theorem~\ref{thm:Siegel}).
Thus assumption~\ref{eqn:assumption} is false.
We deduce that every solution to \eqref{eqn:sunit}
satisfies $m_{\lambda,\mu} \le 2 \ord_\mP(2)$
and in particular satisfies condition (A) of Theorem~\ref{thm:FermatGen}.
This completes the proof.
\end{proof}

The following more recent theorem is from \cite{FKS2}.
\begin{thm}[Freitas, Kraus and Siksek]
Let $\ell \ge 5$ be an odd prime. Suppose $\ell$ is non-Wieferich
(i.e. $2^{\ell-1} \not \equiv 1 \pmod{\ell^2}$). Then the
asymptotic Fermat conjecture holds over $\Q_{n,\ell}$
for all $n \ge 1$.
\end{thm}
A key step towards the proof of this theorem is the following
theorem about unit equations, which applies to $K=\Q_{n,\ell}$
with $\ell \ge 5$.
\begin{thm}
Let $\ell \ge 5$ be an odd prime.
Let $K$ be an $\ell$-extension of $\Q$ (i.e. a finite
Galois extension of $\Q$ with degree $[K:\Q]=\ell^n$
for some $n\ge 1$). Suppose $\ell$ is totally
ramified in $\OO_K$. Then there is no solution to the unit equation in $K$.
%the equation $\lambda+\mu=1$ does not have any solutions   $\lambda$, $\mu \in \OO_{K}^*$.
\end{thm}
\begin{proof}
Let $G=\Gal(K/\Q)$. Let $\mathfrak{L}$ be the unique prime
ideal of $\OO_K$ above $\ell$. As $\ell$ is totally ramified
in $\OO_K$, we know that $\mathfrak{L}^\sigma=\mathfrak{L}$
for all $\sigma \in G$. Moreover, the residue field
$\OO_K/\mathfrak{L}$ is simply $\F_\ell$. In particular,
for any $\lambda \in \OO_K$ then there is some $a \in \Z$
such that $\lambda \equiv a \pmod{\mathfrak{L}}$.
Applying $\sigma \in G$ to this congruence
we see that $\lambda^\sigma \equiv a \pmod{\mathfrak{L}}$.
Let $\Norm$ denote the norm for the extension $K/\Q$.
Then
\[
\Norm(\lambda)=\prod_{\sigma \in G} \lambda^{\sigma}
\equiv a^{\# G} \pmod{\mathfrak{L}}.
\]
Since $\OO_K/\mathfrak{L}=\F_\ell$ and since $\#G=\ell^n$,
Fermat's Little Theorem gives $a^{\#G} \equiv a \equiv \lambda
\pmod{\mathfrak{L}}$.
We deduce that $\Norm(\lambda) \equiv \lambda \pmod{\mathfrak{L}}$
for all $\lambda \in \OO_K$.

Now let $\lambda$, $\mu \in \OO_K^*$ and suppose $\lambda+\mu=1$.
By the above $\lambda \equiv \pm 1 \pmod{\mathfrak{L}}$
and $\mu \equiv \pm 1 \pmod{\mathfrak{L}}$.
Hence $\pm 1 \pm 1 \equiv 1$ in $\OO_K/\mathfrak{L}=\F_\ell$.
This is impossible as $\ell \ge 5$.
\end{proof}

\section{Generalizations}
Let $K$ be a number field (we drop the assumption that
$K$ is totally real).
Let $A$, $B$, $C$ be non-zero elements of $\OO_K$. We consider
the following generalized Fermat equation
\begin{equation}\label{eqn:genferm}
A x^p+B y^p+ C z^p=0,
\end{equation}
and we are interested in solutions $(x,y,z) \in K^3$. We say that
such a solution is \textbf{trivial} if $xyz=0$ otherwise
we say it is \textbf{non-trivial}.
We propose the following generalization of the
asymptotic Fermat conjecture.
%inspired by the works of Deconinck \cite{H} and of
%Kara and Ozman \cite{KO}.
\begin{conj}[A Generalized Asymptotic Fermat Conjecture]
Let $K$ be a number field, and $A$, $B$, $C$ be non-zero
elements of $\OO_K$. Let $\Omega$ be the subgroup of roots
of unity inside $\OO_K^*$. Suppose
\[
A \omega_1+B \omega_2+C \omega_3 \ne 0,
\]
for every $\omega_1$, $\omega_2$, $\omega_3 \in \Omega$.
Then there exists a constant $\mathcal{B}(K,A,B,C)$
such that
for all primes $p>\mathcal{B}(K,A,B,C)$ the only solutions to the Fermat
equation \eqref{eqn:genferm} with $(x,y,z) \in K^3$ are the trivial solutions.
\end{conj}
We point out that this conjecture is a straightforward
consequence of a suitable version
of the $ABC$-conjecture of number fields, such as the one
in \cite{Browkin}.

Equation \eqref{eqn:genferm} with $K=\Q$
was first systematically studied
using the approach via Galois representations and modular
forms by Kraus \cite{Krausppp} and by Halberstadt and Kraus \cite{HK}.
In particular, Halberstadt and Kraus proved the following
remarkable theorem.
\begin{thm}[Halberstadt and Kraus]
Let $A$, $B$, $C$ be odd rational integers. Then for a positive
proportion of primes $p$, the equation \eqref{eqn:genferm}
has no non-trivial solutions $(x,y,z) \in \Z^3$.
\end{thm}
More recently, Dieulefait and Soto \cite{DieulefaitSoto}
have proved a number of theorems
concerning the generalized asymptotic Fermat conjecture, again
with $K=\Q$.
\begin{thm}[Dieulefait and Soto]\label{thm:DS}
Let $A$, $B$, $C$ be rational integers divisible only
by primes $\equiv 1 \pmod{12}$. Then
there is a constant $\cB(A,B,C)$ such that
if $p > \cB(A,B,C)$
then every solution $(x,y,z) \in \Z^3$ to \eqref{eqn:genferm}
is trivial.
\end{thm}
Dieulefait and Soto prove their theorems by reducing to $S$-unit
equations using the same strategy as explained in Section~\ref{sec:SM}.

\bigskip

%Theorem~\ref{thm:FermatGen} has been vastly extended in a number of ways.
Recently a theorem
relating the Fermat equation with coefficients $A x^p+By^p+Cz^p=0$
over totally real fields to $S$-unit equations was proved
by Deconinck \cite{H}. The most general result
is due to Kara and Ozman \cite{KO} which we now describe.
Let $K$ be a number field. We assume
two standard conjectures from the Langlands programme, which
we describe briefly without stating them precisely. For a precise
statement of these conjectures see \cite{KO} or \cite{Haluk}.
\begin{enumerate}
\item[(I)] Serre's modularity conjecture over $K$. This
associates to a totally odd, continuous, finite flat,
absolutely irreducible $2$ dimensional mod $p$
representation  of $\Gal(\overline{K}/K)$
a cuspform of parallel weight $2$ whose
level is equal to the prime-to-$p$ part
of the Artin conductor of the representation.
\item[(II)] An \lq\lq Eichler--Shimura conjecture\rq\rq over $K$.
This associates to a weight $2$ cuspform with rational Hecke eigenvalues
either an elliptic curve or a \lq\lq fake elliptic curve\rq\rq. Note that Conjecture \ref{conj:ES} is a special case of this.
\end{enumerate}
We return to considering \eqref{eqn:genferm}
over a general number field $K$. Let
\[
\mathcal{R}=\prod_{\mathfrak{q} \mid ABC} \mathfrak{q}
\]
where the product is taken over the prime ideals $\mathfrak{q}$ dividing $ABC$.
This is called the \textbf{radical} of $ABC$. Let
\[
S=\{\mP \; : \; \text{$\mP \mid 2 \mathcal{R}$ is a prime ideal of $\OO_K$}\}.
\]
Let
\[
T=\{\mP \; : \; \text{$\mP \mid 2$ is a prime ideal of $\OO_K$},\;
f(\mP/2)=1\}.
\]
%and
%\[
%U=\{ \mP \; : \; $\mP \in T, \; f(\mP/2)=1\}.
%\]
The following is the main theorem of \cite{KO}.
\begin{thm}[Kara and Ozman]\label{thm:KO}
Let $K$ be a number field satisfying conjectures
(I) and (II). Let $A$, $B$, $C$ be odd elements of $\OO_K$
(i.e. $ABC$ is not divisible by any prime ideal $\mP \mid 2$).
Let $S$, $T$ be as above. Suppose that for every
solution $(\lambda,\mu)$ to the $S$-unit equation
\eqref{eqn:sunit} there is a prime $\mP \in T$
such that
\[
\max\{ \lvert \ord_{\mP}(\lambda) \rvert, \; \lvert
\ord_\mP(\mu) \rvert \} \le 4 \ord_\mP(2).
\]
Then the Generalized Asymptotic Fermat's Conjecture holds
for \eqref{eqn:genferm}; in other words there is
a constant $\mathcal{B}(K,A,B,C)$ such that
if $p>\mathcal{B}(K,A,B,C)$ is prime then
the only solutions to \eqref{eqn:genferm}
are the trivial ones.
\end{thm}

We illustrate the theorem of Kara and Ozman by
deriving a slightly stronger version of Theorem~\ref{thm:DS}.
\begin{cor}
Let $\ell$ be an odd prime.
Let $A$, $B$, $C$ be rational integers divisible only
by primes $\equiv \pm 1 \pmod{4\ell}$. Then
there is a constant $\cB(A,B,C)$ such that
if $p > \cB(A,B,C)$
then every solution $(x,y,z) \in \Z^3$ to \eqref{eqn:genferm}
is trivial.
\end{cor}
\begin{proof}
Serre's modularity conjecture over $\Q$ was proved
by Khare and Wintenberger. Over $\Q$ the Eichler--Shimura
conjecture is in fact the Eichler--Shimura theorem.
Thus we can apply Theorem~\ref{thm:KO} unconditionally.
Here, as we're working over $\Z$ we might as well identify prime
ideals with primes. Then
\[
S=\{2\} \cup \{ q_1,q_2,\dotsc,q_r\}, \qquad T=\{2\},
\]
where the $q_i$ are the prime divisors of $ABC$. Thus $q_i \equiv \pm 1
\pmod{4 \ell}$ for $i=1,\dots,r$. Let $(\lambda,\mu)$ be a solution to the
$S$-unit equation $\lambda+\mu=1$.
To deduce the corollary from Theorem~\ref{thm:KO} all we have to do
is to show that
\begin{equation}\label{eqn:ineq}
\lvert \ord_2(\lambda) \rvert \le 4, \qquad \lvert \ord_2(\mu) \rvert \le 4.
\end{equation}
We can rewrite $\lambda+\mu=1$ as
\[
u+v=w, \qquad \lambda=\frac{u}{w}, \qquad \mu=\frac{v}{w},
\]
where
\[
u=\pm 2^{a} \cdot q_1^{\alpha_1} \cdots q_r^{\alpha_r}, \qquad
v=\pm 2^b \cdot q_1^{\beta_1} \cdots q_r^{\beta_r}, \qquad
w= 2^c \cdot q_1^{\beta_1} \cdots q_r^{\beta_r},
\]
where the exponents are non-negative integers,
and we may suppose (after possibly swapping $\lambda$, $\mu$)
that
\begin{enumerate}
\item[(i)]
either $a=b=0$ and $c>0$,
\item[(ii)] or $b=c=0$ and $a>0$.
\end{enumerate}
Let's look at (i). Then
$u \equiv \pm 1 \pmod{4\ell}$ and $v \equiv \pm 1 \pmod{4\ell}$.
Hence
\[
w=u+v \equiv \pm 1 \pm 1 \pmod{4\ell}.
\]
Therefore $w \equiv 2 \pmod{4\ell}$ or $0 \pmod{4\ell}$ or $-2 \pmod{4\ell}$.
However, $\ell \nmid w$ since $\ell \ne q_i$ for $i=1,\dotsc,r$.
Hence $w \equiv \pm 2 \pmod{4\ell}$. Therefore $c=\ord_2(w)=1$.
Hence $\ord_2(\lambda)=a-c=-1$ and $\ord_2(\mu)=b-c=-1$. This establishes
\eqref{eqn:ineq} for case (i). The proof of \eqref{eqn:ineq}
in case (ii) is similar.
\end{proof}

\bigskip

From this Kara and Ozman deduce an analogue
of Theorem~\ref{thm:density} for complex quadratic fields.
\begin{thm}[Kara and Ozman] \label{thm:KOdensity}
Assume conjectures (I) and (II).
Let $\mathbb{N}^{\mathrm{sf}}$ denote the set of squarefree
natural numbers. Let $\mathcal{F}$ be the subset of $d \in \mathbb{N}^{\mathrm{sf}}$
for which the asymptotic Fermat conjecture holds over $\Q(\sqrt{-d})$.
Then
\[
\liminf_{X \rightarrow \infty}
\frac{\# \{ d \in \mathcal{F} \; :\;  d \le X \}}{
\#\{d \in
\mathbb{N}^{\mathrm{sf}} \; : \; d \le X\}} \ge 5/6.
\]
\end{thm}

\bigskip

\noindent \textbf{Remark.}
It is interesting to compare Theorems~\ref{thm:density} and~\ref{thm:KOdensity}.
In the former, the asymptotic Fermat conjecture is established for almost
all real quadratic fields, assuming the Eichler--Shimura conjecture.
In the latter, even assuming the Eichler--Shimura conjecture
and Serre's modularity conjecture, the asymptotic Fermat conjecture is
established for $5/6$ of imaginary quadratic fields.
The reason for the disparity is that the
conclusion of the Eichler--Shimura conjecture
over real quadratic fields is stronger than that for
the Eichler--Shimura conjecture
over complex quadratic fields. Over a real quadratic field $K$
it is conjectured  that
a rational weight $2$ Hilbert eigenform $\ff$ over $K$
corresponds to an elliptic curve $E/K$.
Over a complex quadratic field $K$,
it is conjectured that a rational weight $2$
Bianchi eigenform over $K$ corresponds to
either an elliptic curve $E/K$, or an abelian surface $A/K$
whose
endomorphism algebra is an indefinite division quaternion
algebra (such an abelian surface is called a
\textbf{fake elliptic curve}). If $2$ splits or
ramifies in $K$ then the Frey curve has potentially
multiplicative reduction at the primes above $2$
and it is known that fake elliptic curves have potentially
good reduction at all primes. An image of inertia argument
then allows for the elimination of the fake elliptic curve
case. Unfortunately if $2$ is inert in $K$, then
the Frey elliptic curve might have potentially good
reduction, and we are yet to find a way of
eliminating the possiblity of a fake elliptic curve.
We note that $2$ is inert in $\Q(\sqrt{-d})$
if and only if $-d \equiv 5 \pmod{8}$. This
is $1/6$ of all complex quadratic fields,
and explains the numerical disparity between
Theorems~\ref{thm:density} and~\ref{thm:KOdensity}.

\bigskip

\subsection{Other Signatures}
An  equation of the form $A x^p+By^q=Cz^r$
is called the generalized Fermat equation of \textbf{signature} $(p,q,r)$.
Thus \eqref{eqn:genferm} has signature $(p,p,p)$.
Generalized Fermat equations of signatures $(p,p,2)$
and $(p,p,3)$ have good Frey curves and have been studied, with
$K=\Q$,
respectively by Bennett and Skinner \cite{BennettSkinner} and by Bennett, Vatsal
and Yazdani \cite{BVY}.
%Other families of Fermat type
%equations often considered are
%\[
%A x^p+By^p=Cz^2
%\]
%which has signature
More recently the techniques used by Freitas and Siksek
and by Kara and Ozman have been applied by
Isik, Kara and Ozman \cite{IKO} to study Fermat equations
of signature $(p,p,2)$ over number fields.
\bibliographystyle{plain}
\bibliography{SUnit.bib}
\end{document}